\documentclass[12pt]{amsart}

\usepackage{mathrsfs}
\usepackage{amsmath,amssymb,amsfonts,latexsym,txfonts}     
\usepackage{eucal}
\usepackage{fancyhdr}

\newlength{\originalbase}
\newtheorem{theorem}{Theorem}[section]
\newtheorem{proposition}[theorem]{Proposition}

\input amssym.def
 
\input amssym.tex

\renewcommand{\SS}{\mathbb{S}}
\newcommand{\RR}{\mathbb{R}}

\newcommand{\KK}{\mathscr{K}}

\begin{document}

\title{On the equality conditions of the Brunn-Minkowski theorem}
\author{Daniel A. Klain}
\address{Department of Mathematical Sciences,
University of Massachusetts Lowell,
Lowell, MA 01854 USA
}

\email{Daniel\_{}Klain@uml.edu}

\begin{abstract}
This article describes a new proof of the  
{\em equality} condition for the Brunn-Minkowski inequality.
The Brunn-Minkowski Theorem
asserts that, \sloppy
for compact convex sets $K,L \subseteq \RR^n$,
the $n$-th root of the Euclidean volume $V_n$ is concave with
respect to Minkowski combinations; that is, for $\lambda \in [0,1]$, 
$$V_{n}((1-\lambda)K + \lambda L)^{1/n} \geq 
(1-\lambda) V_{n}(K)^{1/n} + \lambda V_{n}(L)^{1/n}.$$
The equality condition asserts that,
if $K$ and $L$ both have positive volume, then equality holds for some
$\lambda \in (0,1)$
if and only if
$K$ and $L$ are homothetic. 
\end{abstract}

\maketitle

Denote $n$-dimensional Euclidean space by $\RR^n$.  Given compact convex subsets $K,L \subseteq \RR^n$
and $a,b \geq 0$, denote
$$aK + bL = \{ax + by \; | \; x \in K \hbox{ and } y \in L\}.$$
An expression of this form is called a {\em Minkowski combination} or 
{\em Minkowski sum}.  Since $K$ and $L$ are convex sets, the set $aK + bL$ is also convex.  
Convexity also implies that $aK + bK = (a+b)K$ for all $a,b \geq 0$, although this does not hold for general
sets.  Two sets $K$ and $L$ are {\em homothetic} if $K = aL + x$ for some $a > 0$ and some point $x \in \RR^n$.
The $n$-dimensional
(Euclidean) volume of $K$ will be denoted $V_n(K)$.

The Brunn-Minkowski Theorem
asserts that the $n$-th root of the Euclidean volume $V_n$ is concave with
respect to Minkowski combinations; that is, for $\lambda \in [0,1]$, 
\begin{equation}
V_n((1-\lambda)K + \lambda L)^{1/n} \geq (1 - \lambda)V_n(K)^{1/n} + \lambda V_n(L)^{1/n}.
\label{bmcc}
\end{equation}
If $K$ and $L$ have non-empty interiors, then equality holds for some
$\lambda \in (0,1)$
if and only if
$K$ and $L$ are homothetic.  This article describes a new proof of this
{\em equality} condition, using a homothetic projection theorem of Hadwiger.

The Brunn-Minkowski Theorem is the centerpiece of modern convex geometry \cite{Bonn2,Gard2006,red,tony}.  This theorem
encodes as special cases the classical isoperimetric inequality (relating volume and surface area \cite[p. 318]{red}), 
Urysohn's inequality (relating volume and mean width, and strengthening the isodiametric inequality \cite[p. 318]{red}), and 
families of inequalities relating mean projections \cite[p. 333]{red}.  
The concavity implied by~(\ref{bmcc}) leads to families of second-order discriminant-type inequalities for mixed volumes, such as
Minkowski's second mixed volume inequality \cite[p. 317]{red} 
and (after substantial additional labor) the Alexandrov-Fenchel inequality \cite[p. 327]{red}, 
a difficult and far reaching result with consequences in geometric analysis, combinatorics, and algebraic geometry \cite{Ewald}.
Analytic generalizations of the Brunn-Minkowski theorem include the Prekopa-Leindler inequality \cite{Lieb,Hen-Mac}.  
The Brunn-Minkowski theorem also serves as the starting point for analogous developments such as the 
dual Brunn-Minkowski theory for star-shaped sets \cite{Lut1988}, the $L_p$-Brunn-Minkowski theory \cite{bmf1}, 
capacitary Brunn-Minkowski inequalities \cite{Borell}, and Brunn-Minkowski inequalities for integer lattices \cite{Gard-Gron}.
Of special interest are the {\em equality} conditions for~(\ref{bmcc}), which imply, for example, 
the uniqueness of solutions to the Minkowski problem relating
convex bodies to measures on the unit sphere \cite{Bonn2,red}.  
A recent and comprehensive survey on the Brunn-Minkowski inequality and its variations,
applications, extensions, and generalizations, can be found in \cite{Gard-BM}.

There are many ways to prove the Brunn-Minkowski Theorem.  For the case of compact convex sets, 
Kneser and S\"{u}ss used an induction
argument on dimension via slicing \cite[p. 310]{red}.  
This proof first verifies the inequality~(\ref{bmcc}), 
and then addresses the equality case with a more subtle argument.

Hadwiger and Ohmann gave a more general proof
by using the concavity of the geometric mean to prove~(\ref{bmcc}) for case of rectangular boxes, and
following with a divide-and-conquer argument that extends~(\ref{bmcc}) to finite unions of boxes.
They conclude with an approximation step that verifies~(\ref{bmcc}) for all {\em measureable} (including non-convex)
sets \cite{Had-Oh} (see also \cite{Gard-BM} and \cite[p. 297]{Webster}).  
Hadwiger and Ohmann also show that if equality holds in~(\ref{bmcc}) then $K$ and $L$ must both be 
compact convex sets with at most a set of measure zero removed.  Therefore, 
the question of equality in~(\ref{bmcc}) is addressed completely (up to measure zero)
by the case of compact convex sets.  

Another especially intuitive proof
of the inequality~(\ref{bmcc}) for compact convex sets
uses Steiner symmetrization \cite{Gard-BM,Webster}; however, this method relies on approximation and gives no insight into
the equality conditions.

In contrast to earlier methods, the proof of the equality condition for~(\ref{bmcc}) presented in this note 
uses {\em orthogonal projection} rather than slicing.  Sections 1, 2, and 3 provide some technical background.  The new proof
is presented in Section 4.

\section{Background} 

Let $\KK_n$ denote 
the set of compact {\em convex} subsets of $\RR^n$.    If $u$ is a unit vector in
$\RR^n$, denote by $K_u$ the orthogonal projection of a set $K$ onto the subspace $u^\perp$.
More generally, if $\xi$ is a $d$-dimensional subspace of $\RR^n$, denote by $K_\xi$
the orthogonal projection of a set $K$ onto the subspace $\xi$.  The boundary of a compact
convex set $K$ relative to its affine hull will be denoted by $\partial K$.

Let $h_K: \RR^n \rightarrow \RR$ denote the support function of a compact convex set $K$;
that is,
$$h_K(v) = \max_{x \in K} x \cdot v.$$
The standard separation theorems of convex geometry imply that
the support function $h_K$ characterizes the body $K$; that is, $h_K = h_L$ if and only if $K=L$.
If $\xi$ is a subspace of $\RR^n$ then the support function of $K_\xi$ within the subspace 
$\xi$ is given by the restriction
of $h_K$ to $\xi$.   Support functions satisfy the identity
$h_{aK+bL} = ah_K + bh_L$.  (See, for example, any of 
\cite{Bonn2,red,Webster}).

If $u$ is a unit vector in
$\RR^n$, denote by $K^u$ the support set of $K$ in the direction of $u$; that is,
$$K^u = \{x \in K \; | \; x \cdot u = h_K(u) \}.$$
If $P$ is a convex polytope, then $P^u$ is the maximal face of $P$ having $u$ in its outer normal cone.

Given $K, L \in \KK_n$ and $\epsilon > 0$, the function $V_n(K + \epsilon L)$ is a polynomial in $\epsilon$, 
whose coefficients are given by {\em Steiner's formula} \cite{Bonn2,red,Webster}.  
In particular, the following derivative is well-defined:
\begin{align}
nV_{n-1,1}(K,L) & \; = \; \lim_{\epsilon \rightarrow 0} \frac{V_n(K + \epsilon L) - V_n(K)}{\epsilon} 
\; = \; \left. \frac{d}{d \epsilon} \right|_{\epsilon = 0} V_n(K + \epsilon L)
\label{dv}
\end{align}
The expression $V_{n-1,1}$ is an example of a {\em mixed volume} of $K$ and $L$.
Since the volume of any set is invariant under translation, it follows from the definition~(\ref{dv}) that,
for any point $p \in \RR^n$, 
\begin{align*}
V_{n-1,1}(K+p, L) = V_{n-1,1}(K, L+p) = V_{n-1,1}(K, L).
\end{align*}

If $P$ is a polytope, then the mixed volume $V_{n-1, 1}(P, K)$
satisfies the classical ``base-height" formula 
\begin{equation}
V_{n-1, 1}(P, K) = \frac{1}{n} \sum_{u \perp \partial P} h_K(u) V_{n-1}(P^u),
\label{polyvol}
\end{equation}
where this sum is finite, taken over all outer unit normals $u$ to the {\em facets} 
on the boundary $\partial P$.  Since compact convex sets can be approximated
(in the Hausdorff metric) by convex polytopes,
the equation~(\ref{polyvol}) implies that, for $K,L,M \in \KK_n$ and $a,b \geq 0$,
\begin{itemize}
\item $V_{n-1,1}(K, L) \geq 0$,

\item $V_{n-1,1}(K, K) = V_n(K)$,

\item $V_{n-1,1}(aK, bL) = a^{n-1}b V_{n-1,1}(K, L)$

\item $V_{n-1,1}(K, L+M) = V_{n-1,1}(K, L) + V_{n-1,1}(K, M)$,
\end{itemize}
where the final identity follows from~(\ref{polyvol}) and the linearity
of support functions with respect to Minkowski sums.  For $n \geq 3$, the function 
$V_{n-1,1}$ is {\em not} typically symmetric in its parameters:
usually $V_{n-1,1}(K, L) \neq V_{n-1,1}(L, K)$.  In particular,
the fourth (linearity) property above does not typically
hold for the first parameter of $V_{n-1,1}$.  

Important special (and more well-known) cases of mixed volumes result from suitable choices of $K$
or $L$.  For example, if $B$ is the unit ball, centered at the origin, then~(\ref{polyvol}) implies that
$n V_{n-1,1}(P,B)$ gives the {\em surface} area of $P$.  A limiting argument then yields the same fact for
$n V_{n-1,1}(K,B)$, where $K$ is any compact convex set.  Similar arguments imply that, if $\overline{ou}$
denotes the line segment with endpoints at $o$ and a unit vector $u$, then
\begin{align}
n V_{n-1,1}(K,\overline{ou}) = V_{n-1}(K_u),
\label{projvol}
\end{align}
the $(n-1)$-volume of the corresponding orthogonal projection of $K$.
These and many other properties of convex bodies and mixed volumes 
are described in each of \cite{Bonn2,red,Webster}.

One especially intuitive proof
of the Brunn-Minkowski inequality~(\ref{bmcc}) for compact convex sets
uses Steiner symmetrization.  Given a unit vector $u$, 
view $K$ as a family of line segments parallel to $u$.  Slide these segments along $u$ so that
each is symmetrically balanced around the hyperplane $u^\perp$.  By Cavalieri's principle, the volume
of $K$ is unchanged.  Call the new set $st_u(K)$.  It is not difficult to show that 
$st_u(K)$ is also convex, and that $st_u(K) + st_u(L) \subseteq st_u(K+L)$.  A little more work verifies
the following intuitive assertion: if you iterate Steiner symmetrization of $K$ through
a suitable sequence of directions, these iterations tend to round out the body $K$ to a Euclidean ball $B_K$
having the same volume as the original set $K$.  Meanwhile, it follows from the 
aforementioned superadditivity relation that $B_K + B_L \subseteq B_{K+L}$, so that
\begin{align*}
V_n(K + L)^{1/n} &= V_n(B_{K + L})^{1/n} \\
&\geq V_n(B_K + B_L)^{1/n} \\
&= V_n(B_K)^{1/n} + V_n(B_L)^{1/n}\\
&= V_n(K)^{1/n} + V_n(L)^{1/n}
\end{align*}
Technical details behind Steiner symmetrization and the proof outlined above can be found in 
\cite[pp. 306-314]{Webster}.  Once again, because of the approximation step (taking the limit of a sequence of Steiner
symmetrizations), it is not clear from this proof exactly when equality would hold in~(\ref{bmcc}).  

This matter is addressed
by the following theorem.
\begin{theorem}[Minkowski]
If $K$ and $L$ are compact convex sets with non-empty interiors, 
then equality holds in~(\ref{bmcc}) if and only if $K$ and $L$
are homothetic.
\end{theorem}
While homothety
is evidently sufficient, its necessity is far from obvious.

Simple arguments show that the Brunn-Minkowski inequality~(\ref{bmcc}) is equivalent to  
{\em Minkowski mixed volume inequality: }
\begin{equation}
V_{n-1,1}(K,L)^n \geq V_n(K)^{n-1} V_n(L),
\label{mmv}
\end{equation}
where equality conditions are the same as for~(\ref{bmcc}).  

Note that~(\ref{mmv}) is trivial if either $V_n(K)=0$ or $V_n(L)=0$.  Moreover, 
both sides of~(\ref{mmv}) are positively homogeneous of degree $n(n-1)$ with respect to
scaling the body $K$ and positively homogeneous of degree $n$ with respect to
scaling the body $L$.  It follows that, for $K$ and $L$ with non-empty interiors, the
inequality~(\ref{mmv}) is equivalent to the assertion that
\begin{equation}
V_{n-1,1}(K,L)^n \geq 1,
\label{mmv1}
\end{equation}
whenever $V_n(K) = V_n(L) = 1$.

To prove~(\ref{mmv1}), and the equivalent~(\ref{mmv}), using~(\ref{bmcc}), 
suppose that $V_n(K) = V_n(L) = 1$, and let 
$$f(t) \; = \; V_n((1-t)K + tL) \; = \; (1-t)^n V_n \left( K+ \tfrac{t}{1-t}L \right),$$
for $t \in [0,1)$.  By~(\ref{dv}) and the chain rule,
\begin{align}
f'(0) \;  = \; -n V_n(K) + nV_{n-1,1}(K,L) \; = \; -n + nV_{n-1,1}(K,L).
\label{dv0}
\end{align}
Since $f^{1/n}$ is concave by~(\ref{bmcc}), we have $f \geq 1$ on $[0,1)$, while $f(0) = 1$,
so that $f'(0) \geq 0$, and  $V_{n-1,1}(K,L) \geq 1$.

To prove~(\ref{bmcc}) using~(\ref{mmv}) the argument is even simpler.  Denote $K_t =(1-t)K + tL$.
Then
\begin{align*}
V_n(K_t) \, = \, V_{n-1,1}(K_t,K_t) \, & = \,  (1-t) V_{n-1,1}(K_t,K) + t V_{n-1,1}(K_t,L) \\
& \geq \, (1-t) V_n(K_t)^{\frac{n-1}{n}} V(K)^{\frac{1}{n}} 
+ tV_n(K_t)^{\frac{n-1}{n}} V(L)^{\frac{1}{n}}, 
\end{align*}
by two applications of~(\ref{mmv}).  The inequality~(\ref{bmcc}) then follows after division
by  $V_n(K_t)^{\frac{n-1}{n}}$.

For a more complete discussion, see any of \cite{Bonn2,Gard-BM,red,Webster}.

\section{Mixed area}

Denote the special case of $2$-dimensional volume $V_2$ by $A$ for area,
and denote $V_{1,1}(K,L)$ by $A(K,L)$, the {\em mixed area}. 
Unlike the higher dimensional mixed volumes, the mixed area is symmetric in its parameters: 
$A(K,L) = A(L,K)$.
If $\Delta$ is a triangle with outward edge unit normals $u_1, u_2, u_3$, then 
\begin{equation}
A(K,\Delta) = A(\Delta, K) = \frac{1}{2}\sum_i h_K(u_i) |\Delta^{u_i}|,
\label{mah}
\end{equation}
where $|\Delta^{u_i}|$ denotes the length of the $i$-th edge of the triangle.  
This identity leads to the following proposition.
\begin{proposition} Let $K, L \in \KK_2$, and suppose that 
\begin{equation}
A(K, \Delta) = A(L, \Delta) 
\label{mat}
\end{equation}
for every triangle
$\Delta$ in $\RR^2$.  Then $K$ and $L$ are translates.
\label{tri}
\end{proposition}

\begin{proof} Translate $K$ and $L$ so that both lie in the first quadrant of $\RR^2$ and are supported
by the coordinate axes.  (In other words, slide them both ``into the positive corner.")  
If $e_1$ and $e_2$ respectively
denote the unit vectors along the two coordinate axes, we now have
$h_K(-e_1) = h_K(-e_2) = 0$, and similarly for $L$.  Since mixed area is invariant under translation, 
the identity~(\ref{mat}) still holds for every triangle $\Delta$.  

If $u$ is a unit vector with positive coordinates, let $\Delta$ be a right triangle having
outward unit normals $-e_1, -e_2, u$.  Since
$h_K(-e_i) = h_L(-e_i) = 0$, it follows from~(\ref{mah}) and~(\ref{mat}) that $h_K(u) = h_L(u)$.

If $u$ is a unit vector in one of the other 3 quadrants, a similar argument is then made (using a triangle
with unit normals $u$, one of the $-e_i$, and a suitable choice from the first quadrant) to show that
$h_K(u) = h_L(u)$ once again.  It follows that $h_K = h_L$, so that $K=L$ after the initial translations
of $K$ and $L$ into the positive corner.
\end{proof}

\section{Bodies with homothetic projections are homothetic}

In Section~\ref{sec.last} we give a proof of the equality case for~(\ref{bmcc}) and~(\ref{mmv}),
using a projection argument that relies in the following elementary theorem of Hadwiger \cite{Had-proj}.
\begin{theorem}[The Homothetic Projection Theorem] Suppose that $K, L \in \KK_n$ 
have non-empty interiors, where $n \geq 3$.
If $K_u$ and $L_u$ are homothetic for all unit vectors $u$, then
$K$ and $L$ are homothetic as well.
\label{rogers}
\end{theorem}

For completeness of presentation, here is
an elementary proof of Theorem~\ref{rogers} due to Rogers \cite{rogers}.
\begin{proof}   Let
$e_i = (0, \cdots, 1, \cdots, 0)$ denote the unit vector 
having unit $i$th coordinate,
for $i = 1, \ldots n$.

Translate and scale $K$ and $L$ so that both sets are supported by
the positive coordinate halfspace $(e_n^\perp)^+$, and moreover so that 
$K_{e_n} = L_{e_n}$.  The latter is possible, because we are given
that $K_{e_n}$ and $L_{e_n}$ are initially homothetic.

It remains to show that, after these
translations and dilations, we have $K = L$.

To show this, let $u \in e_n^\perp$ be a unit vector.
Recall from the hypothesis of the theorem that 
$K_u = a L_u + v$
for some $a > 0$ and some $v \in u^\perp$.

Let $\xi = \hbox{Span}\{u, e_n\}^\perp = u^\perp \cap e_n^\perp.$
Since $\xi \subseteq e_n^\perp$
and $K_{e_n} = L_{e_n}$,
it follows that
$$L_\xi 
= K_\xi
= a L_\xi + v_\xi .$$
Since $n \geq 3$, we have $\dim \xi = n-2 \geq 1$.
Let $V_{n-2}$ denote volume in $\RR^{n-2}$.  Since translation does
not change volume,
$$V_{n-2} (L_\xi) = 
V_{n-2} (K_\xi) = V_{n-2} (a L_\xi + v_\xi) = 
a^{n-2} V_{n-2}(L_\xi).$$
Since $L$ has non-empty interior, $V_{n-2}(L_\xi) > 0$.
It follows that $a = 1$ and $v_\xi = o$.    
This implies that $v \in \xi^{\perp} = \hbox{Span}\{u, e_n\}$; that is,
$v = bu + ce_n$ for some $b, c \in \RR$.
Moreover $v \cdot u = 0$, since we assumed $v \in u^\perp$
to begin with.  It follows that $v = c e_n$ for some $c \in \RR$.

The positive coordinate halfspace $(e_n^\perp)^+$ 
supports both $K$ and $L$,
so that $h_K(-e_n) = h_L(-e_n) = 0$.  Since 
$K_u = L_u + v$ and $-e_n \in u^\perp$, we have
$$0 = h_K(-e_n) = 
h_{K_u}(-e_n) = h_{L_u}(-e_n) + v \cdot (-e_n) = 
h_L(-e_n) - c = - c, $$
so that $v = ce_n = 0$, and $K_u = L_u$.

We have shown that $K_u = L_u$
for all $u \in e_n^\perp$.  If $v \in \SS^{n-1}$ then 
$v^\perp \cap e_n^\perp \neq \emptyset$, so $v \in u^\perp$
for some $u \in e_n^\perp$.  It now follows that
$$h_K(v) = h_{K_u}(v) = h_{L_u}(v) = h_L(v).$$
In other words, $h_K(v) = h_L(v)$ for all $v \in \SS^{n-1}$, so that $K=L$.
\end{proof}

\section{Conditions for Equality}
\label{sec.last}

We now have the tools to verify
the equality condition for~(\ref{bmcc}) and~(\ref{mmv}).

\begin{proof}[Proof of the Equality Condition]  
For $\lambda \in [0,1]$, denote $K_\lambda = (1-\lambda)K + \lambda L$.
Suppose that equality holds in~(\ref{bmcc}) for some
$\lambda \in (0,1)$, where $K$ and $L$
have non-empty interiors.  
Since equality holds for~(\ref{bmcc}) if and only if equality holds
for~(\ref{mmv}) and~(\ref{mmv1}), the homogeneity of mixed volumes allows us to assume
without loss of generality that 
$$V_n(K) = V_n(L) = V_n(K_\lambda) = V_{n-1,1}(K_\lambda ,L) = 1.$$ 
The concavity of $V_n^{1/n}$ then implies that
$V_n(K_\lambda) = 1$ for {\bf\em all} $\lambda \in [0,1]$.

Fix a value of $\lambda \in [0,1]$, and suppose $M \in \KK_n$ such that $V_{n-1,1}(K_\lambda ,M) \leq 1$.
Since $V_{n-1,1}(\cdot, \cdot)$ is Minkowski linear in its second parameter, we have
$$V_{n-1,1}(K_\lambda ,(1-x)L + xM) = (1-x)V_{n-1,1}(K_\lambda ,L) + xV_{n-1,1}(K_\lambda ,M)
\leq 1$$
for all $x \in [0,1]$.  Since $V_n(K_\lambda) = 1$, it follows from the inequality~(\ref{mmv}) that
\begin{align*}
f(x) & = V_n((1-x)L + xM) \\ & = V_n((1-x)L + xM)\; V_n(K_\lambda)^{n-1} 
\\ & \leq V_{n-1,1}(K_\lambda, (1-x)L + xM)^n \leq 1,
\end{align*}
for all $x \in [0,1]$.  Since $f(0) = V_n(L) = 1$, it follows that $f'(0) \leq 0$. 
On computing $f'(0)$ as in~(\ref{dv0}), we have
$V_{n-1,1}(L, M) \leq 1$.

We have shown that if $V_{n-1,1}(K_\lambda ,M) \leq 1$, then $V_{n-1,1}(L ,M) \leq 1$.
The homogeneity of volume now implies that $V_{n-1,1}(L ,M) \leq V_{n-1,1}(K_\lambda ,M)$
for all $M$.  But the argument above can be repeated, reversing the roles of $K_\lambda$ and $L$.
Therefore,  
$$V_{n-1,1}(K_\lambda ,M) = V_{n-1,1}(L ,M)$$
for all $M$ and all $\lambda \in [0,1]$.

If $n=2$, set $\lambda = 0$.  Proposition~\ref{tri} then implies that $K$ and $L$ are translates.  This case
is the starting point for induction on dimension $n$.

If $n \geq 3$, then assume the theorem holds in lower dimension.
If $u$ is a unit vector, let $M$ denote the line segment $\overline{ou}$,
so that
$$V_{n-1,1}(K_\lambda ,\overline{ou}) = V_{n-1,1}(L , \overline{ou}).$$
It follows from~(\ref{projvol}) that
$$V_{n-1}((K_\lambda)_u) = V_{n-1}(L_u)$$
for all $\lambda \in [0,1]$.  In other words,
$$V_{n-1}((1-\lambda) K_u + \lambda L_u) = V_{n-1}(L_u)$$
for all $\lambda \in [0,1]$.  This is the equality case of the Brunn-Minkowski inequality in dimension $n-1$.
Since $K$ and $L$ have interior, so do their projections (relative to $(n-1)$-dimensional subspaces). 
It now follows from the induction hypothesis that $K_u$ and $L_u$ are homothetic,   
Since $K$ and $L$ have homothetic projections
in every direction $u$, it follows from Theorem~\ref{rogers} that $K$ and $L$ are homothetic.
\end{proof}

\providecommand{\bysame}{\leavevmode\hbox to3em{\hrulefill}\thinspace}
\providecommand{\MR}{\relax\ifhmode\unskip\space\fi MR }
\providecommand{\MRhref}[2]{%
  \href{http://www.ams.org/mathscinet-getitem?mr=#1}{#2}
}
\providecommand{\href}[2]{#2}

\end{document}